\documentclass{RRMPA} 

\Tome{69}
\Year{2024}
\Issue{}
\VolPage{1} 
\CommunicatedBy{xxx xxx}
\Received{...}
\DOI{}

\theoremstyle{thm}
\newtheorem{theorem}{Theorem}[section]
\newtheorem{proposition}[theorem]{Proposition}
\newtheorem{lemma}[theorem]{Lemma}
\newtheorem{corollary}[theorem]{Corollary}

\theoremstyle{def}

\usepackage{enumerate}

\def \R {\mathbb R}

\def \Z {\mathbb Z}

\def\cB{\mathcal{B}}
\def\cC{\mathcal{C}}
\def\cD{\mathcal{D}}

\def\cF{\mathcal{F}}

\def\cH{\mathcal{H}}

\def\cK{\mathcal{K}}
\def\cL{\mathcal{L}}

\def\cN{\mathcal{N}}

\def\cS{\mathcal{S}}

\def\cU{\mathcal{U}}
\def\cV{\mathcal{V}}

\newcommand{\probab}[1]{\ensuremath{\mathbf{P}\big(#1\big)}}
\newcommand{\expect}[1]{\ensuremath{\mathbf{E}\big(#1\big)}}
\newcommand{\var}[1]{\ensuremath{\mathbf{Var}\big(#1\big)}}

\newcommand{\probabom}[1]{\ensuremath{\mathbf{P}_{\omega}\left(#1\right)}}
\newcommand{\expectom}[1]{\ensuremath{\mathbf{E}_{\omega}\left(#1\right)}}
\newcommand{\varom}[1]{\ensuremath{\mathbf{Var}_{\omega}\left(#1\right)}}

\newcommand{\condprobabom}[2]{\ensuremath{\mathbf{P}_{\omega}\left(#1\bigm|#2\right)}}

\newcommand{\ind}[1]{\ensuremath{{1\!\!1}{(#1)}}}

\DeclareMathOperator*{\esssup}{ess\sup}
\DeclareMathOperator*{\wlim}{w\!\lim}

\def\ordo{o}

\def\Ker{\mathrm{Ker}}
\def\Ran{\mathrm{Ran}}
\def\Dom{\mathrm{Dom}}

\newcommand{\abs}[1]{\ensuremath |{#1}|}
\newcommand{\norm}[1]{\ensuremath \|{#1}\|}
\newcommand{\sprod}[2]{\ensuremath \langle{#1,#2}\rangle}

\def \wt {\widetilde}

\def\wh{\widehat}

\begin{document}
\title{Central Limit Theorem for Random Walks in Divergence Free Random Drift Field - Revisited}
\titlerunning{CLT for RW in Divergence Free Random Drift Field} 
\author{B\'alint T\'oth}
\authorrunning{B. T\'oth} 
\affiliation{%
Alfr\'ed R\'enyi Institute of Mathematics, Budapest
\\
\texttt{toth.balint@renyi.hu}

\medskip 
and
\medskip 

School of Mathematics,  University of Bristol
\\
\texttt{balint.toth@bristol.ac.uk}
}

\MakeTitle
\thanks{Supported by the Hungarian National Research and Innovation Office through the grant NKFIH/OTKA K-143468.}

\begin{abstract}
In Kozma-Toth (2017) \cite{kozma-toth-17} the weak CLT was established for random walks in doubly stochastic (or, divergence-free) random environments, under the following conditions:

\smallskip
\noindent
$\circ$
Strict ellipticity assumed for the symmetric part of the drift field. 
\\
$\circ$
$\cH_{-1}$ assumed for the antisymmetric part of the drift field. 

\smallskip
\noindent
The proof relied on a martingale approximation (a la Kipnis-Varadhan) adapted to the \emph{non-self-adjoint} and \emph{non-sectorial} nature of the problem. The two substantial technical components of the proof were:

\smallskip
\noindent
$\star$
A functional analytic statement about the unbounded operator \emph{formally} written as $\abs{L+L^*}^{-1/2}(L-L^*)\abs{L+L^*}^{-1/2}$, where $L$ is the infinitesimal generator of the environment process,  as seen from the position of the moving random walker. 
\\
$\star$
A diagonal heat kernel upper bound which follows directly from Nash's inequality, or, alternatively, from the "evolving sets" arguments of Morris-Peres (2005) \cite{morris-peres-05}, valid only under the assumed \emph{strict ellipticity}.  

\smallskip
\noindent
In this note we present a partly alternative proof of the same result which relies only on functional analytic arguments and \emph{not} on the diagonal heat kernel upper bound provided by Nash's inequality. This alternative proof is relevant since it can be naturally extended to non-elliptic settings pushed to the optimum, which will be presented in  a forthcoming paper. The goal of this note is to present the argument in its simplest and most transparent form. 

\AMSclass{MSC2010: 60F05, 60G99, 60K37} 
\keywords{random walk in random environment (RWRE), divergence-free drift, incompressible flow, central limit theorem} 
\end{abstract}

\section{Introduction}
\label{S:Introduction}

\subsection{Preliminaries}
\label{ss:Preliminaries}

Let $(\Omega, \cF, \pi, (\tau_z:z\in\Z^d))$ be a probability space with an ergodic $\Z^d$-action. Denote by  $\cU:=\{k\in\Z^d: \abs{k}=1\}$ the set of elements of $\Z^d$ neighbouring the origin which will be the set of possible elementary steps of a continuous time nearest neighbour random walk on $\Z^d$.  Let $b: \cU\times\Omega\to[-1,+1]$ be such that 
\begin{align}
\label{flow}
& 
b_k(\omega)+b_{-k}(\tau_k\omega)=0, 
&&
\sum_{k\in\cU} b_k(\omega)=0, 
&&
\int_\Omega b_k(\omega) \, d\pi(\omega)=0.
\end{align}
Thus, the lifted field $b:\cU\times\Z^d\times\Omega\to[-1,+1]$, 
\begin{align*}
b_k(x, \omega):= b_k(\tau_x\omega)
\end{align*}
is a (space-wise) stationary and ergodic, zero-mean, divergence-free flow (or, vector field) on $\Z^d$. 

We study the long-time behaviour of the continuous-time random walk in random environment (RWRE), $t\mapsto X(t)\in \Z^d$ with jump rates
\begin{align}
\label{walk}
\condprobabom{X(t+dt)= x+k}{X(t)=x} = 
\underbrace{(1+b_k(x,\omega))}_{\displaystyle p_k(x, \omega)} dt + \ordo(dt),
\end{align}
and initial position $X(0)=0$. In \eqref{walk}  $p_k(x)$ stands for the jump rate from site $x\in\Z^d$ to the neighbouring site $x+k\in\Z^d$. 

For detailed physical motivation and a collection of concrete examples of the problem we refer to \cite{kozma-toth-17}, \cite{toth-18b}. However, for the reader's convenience we recall concisely some of these in section \ref{ss: Comments, history, examples} below.

We use the notation $\probabom{\cdot}$, $\expectom{\cdot}$  and $\varom{\cdot}$ for \emph{quenched} probability, expectation and variance. That is: probability, expectation, and variance with respect to the distribution of the random walk $X(t)$, \emph{conditionally, with given fixed environment $\omega\in\Omega$}. The notation $\probab{\cdot}:=\int_\Omega\probabom{\cdot} {d}\pi(\omega)$, $\expect{\cdot}:=\int_\Omega\expectom{\cdot} {d}\pi(\omega)$ and $\var{\cdot}:=\int_\Omega\varom{\cdot} {d}\pi(\omega) + \int_\Omega\expectom{\cdot}^2 {d}\pi(\omega) - \expect{\cdot}^2$ will be reserved for \emph{annealed} probability, expectation and variance. That is: probability,  expectation and variance with respect to the random walk trajectory $t\mapsto X(t)$ \emph{and} the environment $\omega$, sampled according to the distribution $\pi$. 

The \emph{environment process} (as seen from the position of the random walker) is, $t\mapsto \eta_t\in\Omega$ defined as 
\begin{align}
\label{envproc}
\eta_t:= \tau_{X(t)}\omega. 
\end{align}
This is a pure jump Markov process on the state space $\Omega$. It is well known (and easy to check, see e.g., \cite{kozlov-85}) that due to the conditions imposed in \eqref{flow} the a priori distribution $\pi$ of the environment is (time-wise) stationary and ergodic for the process $t\mapsto \eta_t\in\Omega$. Hence it follows that the random walk $t\mapsto X(t)$ will have zero-mean stationary and ergodic annealed increments. Though, in the annealed setting the walk is not Markovian. Hence, the strong law 
\begin{align*}
\lim_{t\to\infty} \frac{X(t)}{t}
=0, 
\qquad \text{a.s.}
\end{align*}
obviously follows. Our goal is to establish the CLT
\begin{align}
\label{clt}
t^{-1/2} X(t) \buildrel t\to\infty \over \Rightarrow \cN(0,\sigma^2)
\end{align}
with non-degenerate covariance matrix $\sigma^2$, under suitable assumptions.

\subsection{The $\cH_{-1}$-condition}
\label{ss: The $H_{-1}$-condition}

Beside \eqref{flow} we assume the notorious $\cH_{-1}$-condition holding for the flow field $b$:
\begin{align}
\label{H-1}
\int_{[-\pi,\pi]^d} 
\wh g(p)
\sum_{k\in\cU}
\wh C_{kk}(p)
dp
<\infty. 
\end{align}
where
\begin{align*}
\wh g(p)
:=
\big( \sum_{j=1}^d (1-\cos p_j) \big)^{-1}
\end{align*}
is the Fourier transform of the $\Z^d$-Laplacian's Green-function, and 
\begin{align*}
\wh C_{k,l}(p)
:=
\sum_{x\in \Z^d}
e^{i x\cdot p} 
C_{k,l}(x)
\end{align*}
is the Fourier transform of the correlation of the drift field
\begin{align*}
C_{k,l}(x)
:=
\int_\Omega 
b_k(\omega) b_l(\tau_x\omega) \, d\pi(\omega).
\end{align*}

This is the most natural infrared bound on the decay of correlations of the drift-field $b$. It is well known  (see, e.g. \cite{komorowski-landim-olla-12}, \cite{kozma-toth-17}) that it implies finiteness of the asymptotic variance of $t^{-1/2}X(t)$ on the left hand side of \eqref{clt}, as $t\to\infty$. It is also known that failure of  \eqref{H-1} typically comes with super-diffusive (rather than diffusive) asymptotics of $X(t)$, see, e.g., \cite{komorowski-olla-02}, \cite{toth-valko-12}, \cite{ledger-toth-valko-18}, \cite{cannizzaro-haunschmid-sibitz-toninelli-22}, \cite{chatzigeorgiou-morfe-otto-wang-22}. 

As shown in \cite{kozma-toth-17} (see \cite{komorowski-landim-olla-12} for the continuous space setting) the $\cH_{-1}$-condition \eqref{H-1} is \emph{equivalent} to the following. There exists a function $h:\cU\times\cU\times\Omega\to\R$ such that 
\begin{align}
\label{h-tensor}
&
h_{k,l}(\omega)
=
-h_{-k,l}(\tau_k\omega)
=
-h_{k,-l}(\tau_l\omega)
=
-h_{l,k}(\omega), 
\\[10pt]
\label{h-l2}
&
h_{k,l}\in\cL_2(\Omega,\pi)
\end{align}
and
\begin{align}
\label{v-is-curl-of-h}
&
b_{k}(\omega)
=
\sum_{l\in\cU} h_{k,l}(\omega)
=
\frac12 \sum_{l\in\cU} (h_{k,l}(\omega)-h_{k,l}(\tau_{-l}\omega)).
\end{align}
The second equality in \eqref{v-is-curl-of-h} obviously follows from the symmetries \eqref{h-tensor} of the field $h$. 

Note that all three conditions in \eqref{flow} follow from \eqref{h-tensor}, \eqref{h-l2} and \eqref{v-is-curl-of-h}, which, as shown in \cite{kozma-toth-17}, are (jointly) equivalent to \eqref{H-1}.

The (anti)symmetry conditions \eqref{h-tensor} mean that the lifted field  $h:\cU\times\cU\times\Z^d\times\Omega\to\R$ 
\begin{align*}
h_{k,l}(x,\omega):= h_{k,l}(\tau_x\omega)
\end{align*}
is a translation-wise \emph{ergodic random function of the oriented plaquettes} of $\Z^d$, a.k.a. a (square integrable) \emph{stream tensor}.

\smallskip
\noindent
{\bf Summarizing:} 
We make the \emph{structural} assumptions \eqref{h-tensor}\&\eqref{v-is-curl-of-h} and the \emph{integrability} assumption \eqref{h-l2}. 

\subsection{The CLT}
\label{ss: The CLT}

The standard martingale decomposition of the displacement is
\begin{align}
\label{martingale decomposition}
X(t)=
\underbrace{\big(X(t)-\int_0^t \varphi(\eta_s)ds \big)}_{\displaystyle =: Y(t)} 
+ 
\underbrace{\int_0^t \varphi(\eta_s)ds}_{\displaystyle =: I(t)}
=:
Y(t)+ I(t)
\end{align}
with the drift function $\varphi : \Omega\to\R^d$
\begin{align}
\label{drift}
\varphi(\omega):= \sum_{k\in\cU} k b_k(\omega),
\end{align}
lifted to $\varphi: \Z^d\times\Omega\to \R^d$ as 
\begin{align*}
\varphi(x,\omega):= 
\varphi(\tau_x\omega).
\end{align*}
The  process $t\mapsto Y(t)\in\R^d$ on the right hand side of \eqref{martingale decomposition} is a quenched martingale whose increments are stationary, ergodic and square integrable in the annealed setting. 

In \cite{kozma-toth-17} the Central Limit Theorem \eqref{clt} was established, under optimal (minimal) necessary assumption. 

\begin{theorem}
{\rm (\cite{kozma-toth-17} Theorem 1)}
\label{thm: main}
Assume \eqref{h-tensor}, \eqref{h-l2}, and \eqref{v-is-curl-of-h}. Then the process  $t\mapsto I(t)\in\R^d$ on the right hand side of \eqref{martingale decomposition} is decomposed as 
\begin{align*}
I(t)=Z(t)+E(t)
\end{align*}
so that $t\mapsto Z(t)\in\R^d$ is a quenched martingale whose increments are stationary, ergodic and square integrable in the annealed setting, and 
\begin{align}
\label{error bound}
\lim_{t\to\infty}
t^{-1}\expect{\abs{E(t)}^2}=0. 
\end{align}
The martingales $t\mapsto Y(t)$ and $t\mapsto Z(t)$ do not cancel.
\end{theorem}

\begin{corollary}
\label{cor: main}
Assume \eqref{h-tensor}, \eqref{h-l2}, and \eqref{v-is-curl-of-h}. Then the displacement of the random walk  $t\mapsto X(t)\in\R^d$ is decomposed as 
\begin{align*}
X(t)=
\underbrace{Y(t)+Z(t)}_{\displaystyle =:\wt X(t)}+E(t)
\end{align*}
so that for $\pi$-almost all $\omega$,  under $\probabom{\cdot}$, 
\begin{align*}
N^{-1/2} \wt X(Nt) \Rightarrow \sigma W_\sigma (t), 
\end{align*}
where $t\mapsto W_\sigma(t)$ is a non-degenerate Wiener process on $\R^d$, and the error term $E(t)$ is subdiffusive as shown in \eqref{error bound}. 
\end{corollary}

\noindent
Corollary \ref{cor: main} follows from Theorem \ref{thm: main} by direct application of the Martingale CLT, see, e.g., \cite{mcleish-74}. The proof of Theorem \ref{thm: main} in \cite{kozma-toth-17} relied on two main components: 

\smallskip
\noindent
$\star$
A functional analytic statement about the unbounded operator \emph{formally} written as 
\begin{align*}
B:=
\abs{L+L^*}^{-1/2}(L-L^*)\abs{L+L^*}^{-1/2}, 
\end{align*} 
where $L$ is the infinitesimal generator of the environment process \eqref{envproc}. See \eqref{infgen}, \eqref{L-decomp}, \eqref{B-op-concrete} below. 

\smallskip
\noindent
$\star$
A (quenched) diagonal heat kernel upper bound which follows from Nash's inequality, or, alternatively, from the "evolving sets" arguments of \cite{morris-peres-05}, valid \emph{only} under the assumed \emph{strict ellipticity}.  

\smallskip

In this note we present a partly alternative proof of the same result which relies only on functional analytic arguments and \emph{not} on the diagonal heat kernel upper bound provided by Nash's inequality. This alternative proof is relevant since it can be naturally extended to non-elliptic settings (barred by Nash) pushed to the optimum, which will be presented in  a forthcoming paper. The goal of this note is to present the argument in its simplest and most transparent form. In section \ref{Checking RSC} we present explicitly those details of the proof which differ from \cite{kozma-toth-17}.

\subsection{Comments, history, examples}
\label{ss: Comments, history, examples}

For a comprehensive exposition of the physical motivation, and historic background of the problem we refer to the monograph \cite{komorowski-landim-olla-12}, and the papers \cite{kozma-toth-17}, \cite{toth-18b}. Here we recall very concisely the key facts. Later in this subsection we also recall (informally and succinctly) some concrete examples and counter-examples for the setting \eqref{h-tensor}, \eqref{h-l2}, \eqref{v-is-curl-of-h}.

The continuous space counterpart of the random walk problem considered in this note is the diffusion in random incompressible (or, divergence-free) drift field, $t\mapsto X(t)\in\R^d$, driven by the SDE
\begin{align}
\label{diffusion}
dX(t)
=
b(X(t)) \, dt + \sqrt{2} d W(t), 
\end{align}
where $t\mapsto W(t)$ is a standard Wiener process in $\R^d$ and $x\mapsto b(x)=b(x,\omega)\in\R^d$ is a random vector field over $\R^d$ assumed to be stationary and ergodic with respect to spatial shifts, with finite second moments and zero mean, and almost surely divergence-free:
\begin{align*}
\text{div}\cdot b \equiv0, 
\qquad\text{a.s.}
\end{align*}
The question is formally the same: What are the optimal (minimal) assumptions for the central limit theorem \eqref{clt} to hold.

Motivated by a genuine physical question, namely diffusion of passive tracer particles in steady state, incompressible turbulent flow, 
the random walk and diffusion problems \eqref{walk} and  \eqref{diffusion} have an over forty years long history (spanning from the late 1970-s to the late 2010-s) with considerable effort invested in their satisfactory mathematical understanding. Some of the main stations on this road are (in chronological order) \cite{kozlov-79}, \cite{papanicolaou-varadhan-81}, \cite{osada-83}, \cite{kozlov-85}, \cite{oelschlager-88}, \cite{fannjiang-papanicolaou-96}, \cite{fannjiang-komorowski-97}, \cite{komorowski-olla-03a}, \cite{komorowski-olla-03b}, \cite{deuschel-kosters-08}, \cite{komorowski-landim-olla-12}, \cite{kozma-toth-17}, \cite{toth-18a}. For more details on historic aspects and the results obtained on the way, in the works cited above, see the historic notes in chapters 3 and 11 in the monograph \cite{komorowski-landim-olla-12} and section 1.6 of \cite{toth-18b}. 

Here follow three examples (partly, counterexamples) where conditions  \eqref{h-tensor}, \eqref{h-l2}, \eqref{v-is-curl-of-h} may or may not hold - depending on dimension. We present the examples in $\Z^2$, leaving the (more-or-less obvious) extensions to higher dimensions to the reader. The presentation is verbal and informal. For precise formalisations we refer the reader to section 7 in \cite{kozma-toth-17} and section 1.4 of \cite{toth-18b}.

\medskip

{\sl Example 1: Local rules.} 
This is the baby-version of the basic example from \cite{kozlov-85} where a much more general setting (with finitely dependent drift $((b_k(x)_{k\in\cU})_{x\in\Z^d}$) was exhaustively treated. The faces of the square grid $\Z^2$ are oriented clock-wise or counter-clock-wise independently, with probabilities $\frac12$-$\frac12$. If two neighbouring faces are oriented in opposite sense, then their shared edge gets the orientation dictated by the "consensus" of the two adjecent faces. Otherwise the edge remains unoriented. The random walker is driven by the oriented edges as follows: From any site of $\Z^2$ it jumps to a neighbouring site, 

-
with probability $\frac12$ in the direction of an oriented edge,

-
with probability $0$ opposite to the direction of an oriented edge, 

-
and with probability $\frac14$ along an unoriented edge in either direction.

\noindent
The reader will easily check that due to the construction, these probabilities will always add up to 1, and the drift will be divergence-free in the sense of \eqref{flow}, with the value of $b$ being $+1$,$-1$ and $0$, respectively,  in the three cases listed above.  It is easily seen that conditions \eqref{h-tensor}, \eqref{h-l2}, \eqref{v-is-curl-of-h} hold for this example, and also for its higher dimensional generalizations. Actually, the CLT for a more general class of examples  (with finitely dependent drift $((b_k(x)_{k\in\cU})_{x\in\Z^d}$) was already established in \cite{kozlov-85}.  

\medskip

{\sl Example 2: Randomly oriented Manhattan lattice, and higher dimensional analogs.} 
Orient the horizontal and vertical lines of the square grid $\Z^2$ ("streets", respectively, "avenues") independently of one-another, with probability $\frac12$-$\frac12$, in either one of their two possible directions. All edges on the same (horizontal or vertical) line are oriented in the same direction. The random walker is driven by the oriented edges as follows: From any site of $\Z^2$ it jumps to a neighbouring site 

-
with probability $\frac12$ in the direction of an oriented edge, 

-
and with probability $0$ opposite to the direction of an oriented edge. 

\noindent
The reader will easily check that due to the construction, these probabilities will always add up to 1, and the drift will be divergence-free in the sense of \eqref{flow}, with the value of $b$ being $+1$ and $-1$, respectively,  in the two cases listed above.  Extension to higher dimensions is straightforward. It turns out (see the proof in section 7 of \cite{kozma-toth-17}) that the $H_{-1}$-condition \eqref{H-1} will fail (and thus, there is no representation \eqref{v-is-curl-of-h} of the drift field) in 2 and 3 dimensions, while in dimensions greater than 3 it will hold. Accordingly, in $d\geq 4$ the CLT for the displacement of the random walker will also hold. In \cite{ledger-toth-valko-18} the superdiffusive bounds 
$t^{5/4}\ll \expect{\abs{X(t)}^2}\ll t^{3/2}$ for $d=2$, and 
$t \log\log t \ll \expect{\abs{X(t)}^2}\ll t\log t$ for $d=3$, 
are established (in the sense of Laplace transform, modulo Tauberian inversion), and it is conjectured that 
$\expect{\abs{X(t)}^2}\asymp t^{4/3}$ in $d=2$, and 
$\expect{\abs{X(t)}^2}\asymp t\sqrt{\log t}$ in $d=3$.

\medskip 

{\sl Example 3: The six-vertex (or, square ice) model, and higher dimensional analogs.}
Sample uniformly from all possible configurations of those orientations of all edges of the finite discrete torus $(\Z/L\Z)^2$, where at each single vertex there are exactly two inward and two outward pointing adjecent oriented edges. It is a far from trivial fact that the weak local limit, as $L\to\infty$, exists. This is a "uniformly sampled" random orientation of all edges of the square grid $\Z^2$ with the constraint that at each single vertex there are exactly two inward and two outward pointing adjecent oriented edges. This is the famous and celebrated six-vertex model of lattice statistical physics. In $d$-dimensions, the analogous construction yields the $2d \choose d$-vertex model on $\Z^d$.  The random walker on $\Z^d$ is driven by the oriented edges of the $2d \choose d$-vertex model as follows: From any site of $\Z^d$ it jumps to a neighbouring site 

-
with probability $\frac1d$ in the direction of an oriented edge, 

-
and with probability $0$ opposite to the direction of an oriented edge.

\noindent
The reader will easily check that due to the construction, these probabilities will always add up to 1, 
and the drift will be divergence-free in the sense of \eqref{flow}, with the value of $b$ being $+1$ and $-1$, respectively,  in the two cases listed above. In dimension $d=2$ the $H_{-1}$-condition \eqref{H-1} fails (just marginally, with a logarithmic divergence), while in dimensions $d\geq3$ it holds. As a consequence, the central limit theorem holds for the random walker on the $2d \choose d$-vertex model on $\Z^d$, in $d\geq3$ and presumably fails in $d=2$. In $d=2$ the superdiffusive asymptotics $\expect{\abs{X(t)}^2}\asymp t\sqrt{\log t}$ is conjectured (but far from proved), cf. \cite{toth-valko-12}.

\section{Spaces and operators}
\label{s: Spaces and operators}

\subsection{The infinitesimal generator}
\label{ss: The infinitesimal generator}

The infinitesimal generator of the environment process $t\mapsto \eta_t$ \eqref{envproc} is 
\begin{align}
\label{infgen}
Lf(\omega)
&
=
\sum_{k\in\cU} 
\underbrace{(1+ b_k(\omega))}_{p_k(\omega)}(f(\tau_k\omega)-f(\omega)).
\end{align}
This operator is well defined acting on all measurable functions $f:\Omega\to \R$ 

It is decomposed into Hermitian and anti-Hermitian parts (w.r.t. the stationary measure $\pi$ ) as 
\begin{align}
\label{L-decomp}
L
=-S+A, 
&&
Sf(\omega)
:=
-
\!\!\!
\sum_{k\in\cU} (f(\tau_k\omega)\!-\!f(\omega)), 
&&
Af(\omega)
:=
\sum_{k\in\cU} b_k(\omega)(f(\tau_k\omega)\!-\!f(\omega)) 
\end{align}

\subsection{Basic spaces and operators}
\label{ss: Basics}

We define various function spaces (over  $(\Omega, \pi)$) and linear operators acting on them. With usual abuse we denote \emph{classes of equivalence} of $\pi$-a.s. equal measurable functions simply as functions. Let the space of \emph{scalar}-, \emph{vector}-, \emph{rotation-free vector}-, and \emph{divergence-free vector} fields be 
\begin{align*}
\cS
&
:=
\{f:\Omega\to\R: 
f \text{ is } \cF\text{-measurable}\}
\\[10pt]
\cV
&
:=
\{ u:\Omega\to\R^\cU: 
u_k\in \cS, \ \ 
u_k(\omega)+u_{-k}(\tau_k\omega)=0, \ \ k\in\cU, 
\ \ \pi\text{-a.s.}\}
\\[10pt]
\cK
&
:=
\{u\in\cV: 
u_k(\omega)+u_l(\tau_k\omega)=u_l(\omega)+u_k(\tau_l\omega),  \ \ 
k,l\in\cU,  \ \ \pi\text{-a.s.}\}
\\[10pt]
\cD
&
:=
\{u\in\cV: 
\sum_{k\in\cU}u_k(\omega)=0
 \ \ \pi\text{-a.s.}\}.
\end{align*}
These are linear spaces (over $\R$) with no norm or topology endowed on them  yet. We call these spaces these names for the obvious reason that their lifting
\begin{align*}
&
f(x,\omega):= f(\tau_x\omega)
\ \ \ 
(f\in\cS),
&&
\text{respectively,}
&&
u_k(x,\omega):= u_k(\tau_x\omega)
\ \ \ 
(u\in\cV),
\end{align*}
are translation-wise ergodic scalar, respectively, vector fields over $\Z^d$. 

The linear operators $\partial_k, H_{k,l}: \cS\to\cS$, $k,l\in\cU$, defined below  on the whole space $\cS$ as their domain, will be the basic building blocks used in constructing more complex operators.
\begin{align}
\notag
\partial_k f(\omega)
&
:=
f(\tau_k\omega) - f(\omega), 
\\[10pt]
\label{Hkl-op}
H_{k,l}f(\omega)
&
:=
h_{k,l}(\omega) f(\omega)
\end{align}
Using these basic operators we furhter define
\begin{align}
\notag
&
\nabla:\cS\to\cV,
&&
(\nabla f)_k
:=
\partial_k f
\\[10pt]
\notag
&
\nabla^*:\cV\to \cS, 
&&
\nabla^* u
:= 
\sum_{k\in\cU}u_k
=
-
\frac12
\sum_{k\in\cU} \partial_{-k} u_k
\\[10pt]
\notag
&
\Delta:\cS\to\cS
&&
\Delta
:=
-\nabla^*\nabla 
=
\sum_{k\in\cU}
(\partial_k -I)
\\[10pt]
\label{H-op}
&
H:\cV\to\cV
&&
(Hu)_k
:=
\frac12
\sum_{l\in\cU} 
H_{k,l}(\partial_k + 2 I ) u_l
\end{align}
These operators are well defined on the whole spaces given as their respective domains and obviously, 
\begin{align*}
\Ran(\nabla)\subset\cK, 
\qquad
\Ker (\nabla^*)
=
\cD.
\end{align*}
For the time being the superscript $^*$ is only notation. It will later indicate adjunction with respect to the inner products defined in \eqref{scalar-products} below. 

On the right hand side of \eqref{H-op} the term $(\partial_k+2I)/2$ takes care of projecting back to $\cV$. This is necessary and important. One can easily check that 
\begin{align*}
&
\text{for any } 
w\in \cK: 
&&
\sum_{k\in\cU} 
b_k w_k
=
\nabla^* H w, 
\end{align*}
and this identity holds \emph{only for} $w\in \cK\subsetneqq\cV$. Using this identity, 
the Hermitian and anti-Hermitian parts of the infinitesimal generator $L$, defined in \eqref{L-decomp} are written as
\begin{align*}
&
S
=
-\Delta
=
\nabla^* \nabla, 
&&
A
=
\nabla^* H\nabla 
\end{align*}
Basically, we will work in the real Hilbert spaces
\begin{align*}
&
\cS_2
:=
\{f\in\cS: \norm{f}_2^2:= \int_{\Omega}\abs{f(\omega)}^2\, d\pi(\omega)<\infty, \ \ \int_\Omega f(\omega)\, d\pi(\omega)=0\},
\\[10pt]
&
\cV_2
:=
\{u\in\cV: \norm{u}_2^2:= \frac12\sum_{k\in\cU} \norm{u_k}_2^2 < \infty,  \ \ \int_\Omega u(\omega)\, d\pi(\omega)=0\},
\\[10pt]
&
\cK_2
:=
\cK\cap\cV_2,
\qquad
\cD_2
:=
\cD\cap\cV_2,
\end{align*}
with the scalar products
\begin{align}
\label{scalar-products}
&
\sprod{f}{g}
:=
\int_\Omega f(\omega)\, g(\omega)\, d\pi(\omega), 
&&
\sprod{u}{v}
:=
\frac12 \sum_{k\in\cU} \sprod{u_k}{b_k}. 
\end{align}
(We don't introduce  different notation for the norms and scalar products in $\cS_2$, respectively, $\cV_2$. The precise meaning of $\norm{\cdot}_2$ and $\sprod{\cdot}{\cdot}$ will be always clear from the context.)

Due to ergodicity of $(\Omega, \cF, \pi, \tau_z:z\in\Z^d)$, the space of square integrable vector fields $\cV_2$ is orthogonally decomposed as
\begin{align*}
\cV_2
=
\cK_2 \oplus \cD_2. 
\end{align*}
Later we will denote by $\Pi:\cV_2\to\cK_2$ the orthogonal projection from $\cV_2$ to $\cK_2$, see \eqref{orthproj}. 

In section \ref{Checking RSC} we will also need the Banach spaces
\begin{align*}
&
\cS_\infty
:=
\{f\in\cS: 
\norm{f}_\infty
:= 
\esssup_{\omega\in(\Omega,\mu)} \abs{f(\omega)}<\infty, 
\ \ \int_\Omega f(\omega)\, d\pi(\omega)=0\},
\\[10pt]
&
\cV_\infty
:=
\{u\in\cV: 
\norm{u}_\infty
:= 
\max_{k\in\cU} \esssup_{\omega\in(\Omega,\mu)} \abs{u_k(\omega)}<\infty, 
\ \ \int_\Omega u(\omega)\, d\pi(\omega)=0\},
\\[10pt]
&
\cK_\infty
:=
\cK\cap\cV_\infty,
\qquad
\cD_\infty
:=
\cD\cap\cV_\infty,
\end{align*}

The operators $\partial_k:\cS_2\to\cS_2$, $\nabla:\cS_2\to\cK_2$, $\nabla^*:\cV_2\to\cS_2$, $\Delta:\cS_2\to\cS_2$ are bounded, and their adjointness relations (with respect to the scalar products \eqref{scalar-products}) are  obviously
\begin{align*}
&
\partial_k^{*}=\partial_{-k}
&& 
(\nabla)^*=\nabla^*
&&
\Delta^*=\Delta\leq0. 
\end{align*}
The operators $H_{k,l}$ and $H$ defined in \eqref{Hkl-op}, respectively, \eqref{H-op}, when restricted to $\cS_2$, respectively, to $\cV_2$, are \emph{unbounded} with respect to the norms $\norm{\cdot}_2$. However, as multiplication operators there is no issue with their proper definition as densely defined self-adjoint, respectively, skew-self-adjoint operators: 
\begin{align*}
&
H_{k,l}^*=H_{k,l}
&& 
H^{*}=-H.
\end{align*}

\subsection{The space  $\cH_-$ and the Riesz operators}
\label{ss: "Sobolev spaces"}

As $\Delta=\Delta^*\leq0$ we define the self-adjoint operators $\abs{\Delta}^{1/2}=(-\Delta)^{1/2}$ and $\abs{\Delta}^{-1/2}=(-\Delta)^{-1/2}$ in terms of the Spectral Theorem, and the subspace 
\begin{align}
\notag
\cH_{-}
& 
:=
\big\{
f\in\cS_2: 
\norm{f}_-^2
:=
\lim_{\lambda\searrow0} \sprod{f}{(\lambda I-\Delta)^{-1}f}
=
\norm{\abs{\Delta}^{-1/2} f}_2^2<\infty
\big\}
\\[10pt]
\label{H-space}
& 
\phantom{:}
= 
\Dom(\abs{\Delta}^{-1/2})
= 
\Ran(\abs{\Delta}^{1/2}). 
\end{align}
Since $\Delta$ is a bounded operator over $(\cS_2, \norm{\cdot}_2)$, the Euclidean space  $(\cH_-, \norm{\cdot}_-)$ is a complete Hilbert space (closed in the $\norm{\cdot}_-$-norm, as defined in \eqref{H-space}), and since $0$ is a nondegenerate eigenvalue of $\Delta$ (due to ergodicity of $(\Omega, \pi, \tau_z:z\in\Z^d)$), $\cH_-$ is a dense subspace of $(\cS_2, \norm{\cdot}_2)$. 

Next, we define the \emph{Riesz operators} 
\begin{align}
\label{Riesz-ops}
&
\Lambda:= \nabla \abs{\Delta}^{-1/2}:\cS_2\to \cK_2, 
&&
\Lambda^*:= \abs{\Delta}^{-1/2} \nabla^*: \cV_2 \to \cS_2. 
\end{align}
It is obvious that 
\begin{align*}
& 
\norm{\Lambda f}_2 = \norm{f}_2 \
\text{ for }
f\in \cS_2,
&&
\Ker(\Lambda^*)=\cD, 
&&
\norm{\Lambda^* u}_2 = \norm{u}_2 \
\text{ for }
u\in \cK_2. 
\end{align*}
(More pedantically, a priori $\Lambda:\cH_-\to\cK_2$ extends to $\Lambda:\cS_2\to\cK_2$ as an isometry.)
Finally, we also have 
\begin{align}
\label{orthproj}
&
\Lambda^*\Lambda = I_{\cS_2}, 
&&
\Pi:= \Lambda\Lambda^*: \cV_2\to\cK_2 
\end{align}
The latter being the orthogonal projection from $\cV_2$ to $\cK_2$.

\section{Proof of Theorem \ref{thm: main}}
\label{s: Proof of Theorem {main}}

\subsection{Kipnis-Varadhan theory - the abstract form}
\label{ss: Kipnis-Varadhan}

The proof of Theorem \ref{thm: main} is based on the non-reversible (i.e. non-self-adjoint) and non-sectorial version of martingale approximation a la Kipnis-Varadhan, summarized concisely in this section.

Let $(\Omega, \cF, \pi)$ be a probability space and $t\mapsto\eta_t\in\Omega$ a Markov process assumed to be stationary and ergodic under the probability measure $\pi$, whose infinitesimal generator acting on $\cL_2(\Omega, \pi)$ decomposes as 
\begin{align*}
&
L=-S+A, 
&&
S:=-(L+L^*)/2, 
&&
A:=(L-L^*)/2. 
\end{align*}
and whose resolvent is denoted
\begin{align*}
R_\lambda:= (\lambda I-L)^{-1}. 
\end{align*}

For our current purpose we can (somewhat restrictively) assume that the operators $L,S,A$ are \emph{bounded} and also that the self-adjoint part $S$ is ergodic on its own. That is: $Sf=0$ if and only if $f$ is constant. With this in view, we restrict all computations to the subspace  of codimension 1
\begin{align*}
\cL_{2,0}
:=
\{f\in\cL_2: \int_\Omega f \, d\pi=0\}.
\end{align*}
(This corresponds to the Hilbert space $\cS_2$ in the concrete setting of our problem.)

Finally, we'll also need the subspace 
\begin{align*}
\notag 
\cH_-
&
:=
\{f\in\cL_{2,0}: \norm{f}_-^2:=\lim_{\lambda\searrow 0}\sprod{f}{(\lambda I+S)^{-1}f}=\norm{S^{-1/2}f}^2<\infty\}
\\[10pt]
&
\phantom{:}
= \Dom(S^{-1/2})
= \Ran(S^{1/2}), 
\end{align*}
with the operators $S^{\pm1/2}$ defined in terms of the Spectral Theorem.  

We quote from \cite{toth-86,horvath-toth-veto-12, toth-13}  the Kipnis-Varadhan martingale approximation in the non-self-adjoint setting. See the monograph \cite{komorowski-landim-olla-12} for historic background. 

\begin{theorem}
{\rm (\cite{toth-86,horvath-toth-veto-12, toth-13} Theorem KV)}
\label{thm: toth-86}
Let $\varphi\in\cL_{2,0}(\Omega,\pi)$. If the following two conditions hold
\begin{align}
\label{condition A&B}
& 
\lim_{\lambda\to0} 
\lambda^{1/2}\norm{R_\lambda \varphi}_2=0, 
&&
\lim_{\lambda\to0} 
\norm{S^{1/2}R_\lambda\varphi-v}_2=0, \ \ v\in\cL_2,  
\end{align}
then 
\begin{align*}
\sigma^{2}
:=
2
\lim_{\lambda\to0} 
\sprod{\varphi}{R_\lambda\varphi}
=
2
\norm{v}_2^2
\in[0,\infty)
\quad
\text{exists,}
\end{align*}
and there exists an $\cL_2$-martingale $t\mapsto Z(t)$, with stationary and
ergodic increments, adapted to the natural filtration $\cF_t$ of the Markov process $t\mapsto\eta_t$, and  with variance
\begin{align*}
\expect{\abs{Z(t)}^2}=\sigma^2 t,
\end{align*}
such that 
\begin{align}
\label{KV-martingale-approximation}
\lim_{t\to\infty}
t^{-1} \expect{\abs{\int_0^t \varphi(\eta_s)\, ds - Z(t)}^2}=0. 
\end{align}
\end{theorem}

\noindent
Conditions \eqref{condition A&B} of Theorem \ref{thm: toth-86} are difficult to check directly. Sufficient conditions are known under the names of Strong Sector Condition \cite{varadhan-95}, respectively,  Graded Sector Condition \cite{sethuraman-varadhan-yau-00}. See also the monograph \cite{komorowski-landim-olla-12} for context and details. However, these conditions hold only under very special assumptions about the Markov process considered: a graded structure of the infinitesimal generator $L$ acting on an accordingly graded Hilbert space $\cL_2$. This structural assumption simply doesn't hold in many cases of interest, including our current problem. 

The next theorem, quoted from \cite{horvath-toth-veto-12}, provides a sufficient condition which does not assume sectorial structure (grading) of the infinitesimal generator $L$ acting on the Hilbert space $\cL_{2,0}(\Omega,\pi)$. Let
\begin{align*}
\cB:=
\{f\in\cH_-: 
AS^{-1/2}f
\in\cH_-
\}
\end{align*}
and $B:\cB\to\cL_{2,0}$ defined as
\begin{align*}
Bf
:=
S^{-1/2}A S^{-1/2}f. 
\end{align*}
Note that the operator $B:\cB\to\cL_{2,0}$ is unbounded (except for the elementary cases when the operator $S:\cL_{2,0}\to \cL_{2,0}$ is invertible) and  skew symmetric. Indeed, for $f,g\in\cB$ all the straightforward steps below are legitimate
\begin{align*}
\sprod{f}{S^{-1/2}AS^{-1/2}g}
&
=
\sprod{S^{-1/2}f}{AS^{-1/2}g}
\\[10pt]
&
=
-
\sprod{AS^{-1/2}f}{S^{-1/2}g}
=
-
\sprod{S^{-1/2}AS^{-1/2}f}{g}
\end{align*}
Of course, it could happen that the subspace $\cB$ is not dense in $\cL_{2,0}$, or, even worse, that simply $\cB=\{0\}$. Even if $\cB$ is a dense subspace  in $\cL_{2,0}$, in principle it could still happen that the operator $B$ (which in this case is densely defined and skew-symmetric) is not essentially skew-self-adjoint.

\begin{theorem} 
{\rm (\cite{horvath-toth-veto-12} Theorem 1)}
\label{thm: rsc}
Assume that there exists a subspace $\cC\subseteqq\cB$ which is dense in $\cL_{2,0}$ and the operator $B:\cC\to\cL_2$ is essentially skew-self-adjoint (that is,  $\overline B=-B^*$).  Then for any $\varphi\in\cH_-$ the conditions of Theorem \ref{thm: toth-86} (and hence the martingale approximation \eqref{KV-martingale-approximation}) hold.
\end{theorem}

\noindent
{\bf Remarks:} 
(1) 
In \cite{horvath-toth-veto-12} the theorem is formulated in slightly different terms. However, it is easy to see that this form follows directly from that of Theorem 1 in \cite{horvath-toth-veto-12}. 
\\
(2)
The conditions of Theorem \ref{thm: rsc} are equivalent to $\cB$ being dense in $\cL_{2,0}$ and $B:\cB\to\cL_2$ essentially skew-self-adjoint. The formulation of the theorem gives some flexibility in choosing the core $\cC\subseteqq \cB$. 

\subsection{Proof of Theorem \ref{thm: main}}
\label{Checking RSC}

We check the conditions of Theorem \ref{thm: rsc} for the concrete case under consideration, when the Markov process $t\mapsto\eta_t$  is the environment process cf. \eqref{envproc}, its infinitesimal generator $L$ given in \eqref{infgen}, \eqref{L-decomp} acts on the Hilbert space $\cS_2$, and $\varphi$ is the drift given in \eqref{drift}.

It is essentially straightforward (and shown in \cite{kozma-toth-17}) that the $H_{-1}$-condition \eqref{H-1} (and thus also \eqref{h-tensor}, \eqref{h-l2}, \eqref{v-is-curl-of-h}, jointly) are equivalent to $\varphi\in\cH_{-1}$. (Hence the name of the condition \eqref{H-1}.)
It remains to prove skew-self-adjointness of the operator $S^{-1/2}AS^{-1/2}=\abs{\Delta}^{-1/2}\nabla^* H \nabla \abs{\Delta}^{-1/2}$ -- properly defined. This is exactly what we do in what follows. 

In this case the subspace $\cB$ is 
\begin{align*}
\cB
:=
\{
f
\in\cH_-: 
\nabla^*H\nabla\abs{\Delta}^{-1/2} f
\in\cH_-
\}
\end{align*}
and $B:\cB\to\cS_2$ acts as
\begin{align}
\label{B-op-concrete}
Bf:= 
\abs{\Delta}^{-1/2} \nabla^* H \nabla \abs{\Delta}^{-1/2} f. 
\end{align}
Let 
\begin{align*}
\cC
&
:=
\{f=\abs{\Delta}^{1/2} g: g \in \cS_\infty\}
\end{align*}
Obviously, the subspace $\cC$  is dense in $(\cS_2, \norm{\cdot}_2)$, and for $f\in\cC$, the equation $f=\abs{\Delta}^{1/2}g$ determines \emph{uniquely} $g\in\cS_\infty$. Furthermore, 
\begin{align*}
\nabla^*H\nabla\abs{\Delta}^{-1/2} f
=
\nabla^ *
\underbrace{H 
\underbrace{\nabla 
\underbrace{g}_{\in\cS_\infty}}_{\in\cK_\infty}}_{\in\cV_2}
\in\cH_-. 
\end{align*}
Thus, indeed, 
\begin{align*}
\cC\subsetneqq\cB\subsetneqq\cS_2=\overline{\cC}, 
\end{align*}
where the $\overline{\cC}$ denotes closure of $\cC$ with respect to the norm $\norm{\cdot}_2$. 

\begin{proposition}
\label{prop: B-is-skew-self-adjoint}
The linear operator $B:\cC\to\cS_2$ is essentially skew-self-adjoint.  
\end{proposition}

\begin{proof}
Let 
\begin{align*}
\cF
:=
\Lambda \cC
=
\{\nabla g: g\in\cS_\infty\}
\subsetneqq 
\cK_\infty
\subsetneqq
\cK_2, 
\end{align*}
and define the operator $F:= \cF\to \cK_2$ as
\begin{align}
\label{op-F-def}
&
F:= 
\Lambda B\Lambda^* 
= 
\Pi H \Pi, 
&&
Fu
:=
\underbrace{\Pi \underbrace{H \underbrace{u}_{\in\cK_\infty}}_{\in\cV_2}}_{\in\cK_2}, 
\end{align}
where $\Lambda, \Lambda^*$ and $\Pi$ are the Riesz operators and the orthogonal projection defined in \eqref{Riesz-ops}, \eqref{orthproj}. 

Since $\Lambda: \cS_2\to \cK_2$, $\Lambda^*:\cK_2\to\cS_2$ are \emph{isometries}, the statement of the proposition is equivalent to the operator $F:\cF\to\cK_2$ being essentially skew-self-adjoint. This is what we are going to prove. 

Obviously, the operator $F$ is  skew-symmetric on $\cF$, since for $u,w \in \cF$ the identity 
\begin{align*}
\sprod{ u}{ H v}
=
-
\sprod{H  u}{ v}
\end{align*}
is legitimate. Next we define the adjoint of  $F: \cF\to \cK_2$. Its domain is 
\begin{align*}
\cF^{*}
:=
\{w\in\cK_2 : \exists c=c(w)<\infty : \forall u\in\cF: \abs{\sprod{w}{ H u}}\leq c \norm{u}_2 \},
\end{align*}
and $F^*:\cF^*\to \cK_2$ is defined uniquely by the Riesz Lemma: for any $w\in\cF^*$, $F^* w$ is the unique element of $\cK_2$ such that for all $u\in \cF$
\begin{align*}
\sprod{F^* w }{u} 
=
\sprod{w}{ H u}. 
\end{align*}
Obviously, $\cF\subsetneqq \cF^*$,  $F^*|_{\cF}=-F$, and 
\begin{align*}
F 
\prec 
F^{**}
\preceq
-F^{*}
\end{align*}
In order to conclude 
\begin{align*}
F^{**}
=
-F^{*}
\end{align*}
and thus essential skew-self-adjointness of $F$, as defined in \eqref{op-F-def}, it is sufficient to prove that $F^*$ is skew-symmetric on $\cF^*$. 

For $K<\infty$, let 
\begin{align*}
h^{K}_{k,l}(\omega)
:=
h_{k,l}(\omega) \ind{\abs{h_{k,l}(\omega)}\leq K}.
\end{align*}
These truncated functions inherit the stream-tensor (anti)symmetries \eqref{h-tensor}. We define the \emph{bounded} operator $H^{K}: \cV_2\to\cV_2$ by \eqref{Hkl-op}\&\eqref{H-op}, with the stream tensor $h$ replaced by its truncated version $h^K$

\begin{lemma}
\label{lem: weak-limit}
For $w\in\cF^*$ 
\begin{align}
\label{weak-lim}
F^*w 
=
-
\wlim_{K\to\infty}  \Pi H^{K} w,
\end{align}
where $\wlim$ denotes weak limit in the Hilbert space $(\cK_2, \norm{\cdot}_2)$.
\end{lemma}

\begin{proof}
This is straightforward. Let $w\in\cF^*$ and $u\in\cF$. Then 
\begin{align*}
-
\lim_{K\to\infty}
\sprod{ H^{K} w}{u}
=
\lim_{K\to\infty}
\sprod{w}{ H^{K} u}
=
\sprod{w}{ H u}
=
\sprod{F^* w}{u}, 
\end{align*}
where the limit in the second step follows from \emph{uniform integrability} and \emph{almost sure convergence} (over $(\Omega, \pi)$) of the sequence of functions 
\begin{align*}
\omega\mapsto
\sum_{k,l \in\cH} w_k(\omega) h^{K}_{k,l}(\omega) (u_l(\omega)+u_l(\tau_k\omega))
\end{align*}
as $K\to\infty$. 

Since $\cF$ is dense in $\cK_2$ \eqref{weak-lim} follows. 
\end{proof}

From \eqref{weak-lim} the skew-symmetry of the operator $F^*:\cF^*\to\cK_2$ drops out: for  $u,w\in\cF^*$
\begin{align*}
\sprod{w}{F^* u}
=
\lim_{K\to\infty}
\sprod{w}{H^K u}
=
-
\lim_{K\to\infty}
\sprod{H^K w}{ u}
=
-
\sprod{F^* w}{ u}.
\end{align*}
This concludes the proof of essential skew-self-adjointness of the operator $F:\cF\to \cK_2$ and thus of the operator $B$ of \eqref{B-op-concrete} defined on the core $\cC=\Lambda^* \cF$. 
\end{proof}
This also concludes checking all conditions of Theorem \ref{thm: rsc} in the concrete setting and thus also the proof of Theorem \ref{thm: main}. 
\qed

\bigskip

\begin{acknowledgements} 
This work was supported by the Hungarian National Research and Innovation Office through the grant NKFIH/OTKA K-143468. 
\end{acknowledgements}


\end{document}